\documentclass[a4paper,10pt]{amsart}
\usepackage{amssymb}
\usepackage{amsmath}
\usepackage{todonotes}
\usepackage{esint}
\usepackage{enumitem}
\usepackage{mathrsfs}
\usepackage{mathtools}
\usepackage{color}
\definecolor{blue3}{RGB}{0,0,205}
\usepackage[colorlinks,citecolor=blue3,linkcolor=blue3,urlcolor=blue3,pdfpagemode=UseNone]{hyperref}

\newtheorem{thm}{Theorem}[section]
\newtheorem{corollary}[thm]{Corollary}
\newtheorem{lemma}[thm]{Lemma}
\newtheorem{proposition}[thm]{Proposition}

\theoremstyle{definition}
\newtheorem{definition}[thm]{Definition}
\newtheorem{example}[thm]{Example}

\theoremstyle{remark}
\newtheorem{remark}[thm]{Remark}

\newcommand{\cF}{{\mathcal F}}

\newcommand{\ve}{\varepsilon}

\newcommand{\bC}{{\mathbb C}}

\newcommand{\bN}{{\mathbb N}}

\newcommand{\bR}{{\mathbb R}}
\newcommand{\bZ}{{\mathbb Z}}

\newcommand{\oL}[1]{\bar{#1}}

\newcommand{\fg}{\mathfrak{g}}

\newcommand{\fn}{\mathfrak{n}}

\newcommand{\ft}{\mathfrak{t}}

\newcommand{\bX}{\bar{X}}

\newcommand{\Ad}{\operatorname{Ad}}

\newcommand{\id}{\operatorname{id}}

\numberwithin{equation}{section}

\begin{document}

\title{Toric actions in cosymplectic geometry}

\author{Giovanni Bazzoni and Oliver Goertsches}
\address{Departamento de \'Algebra, Geometr\'ia y Topolog\'ia, Facultad de Ciencias Matem\'a\-ticas, Universidad Complutense de Madrid, Plaza de Ciencias 3, 28040 Madrid}
\email{gbazzoni@ucm.es}

\address{Fachbereich Mathematik und Informatik, Philipps-Universit\"at Marburg, Hans-Meer\-wein-Stra\ss e 6, 35032 Marburg}
\email{goertsch@mathematik.uni-marburg.de}

\date{}


\begin{abstract}
We show that compact toric cosymplectic manifolds are mapping tori of equivariant symplectomorphisms of toric symplectic manifolds.
\end{abstract}

\maketitle
\section{Introduction}

The study of smooth manifolds endowed with the action of a Lie group dates back to almost 150 years ago, with the \emph{Erlanger Programm} of Felix Klein. In its general form, it is a very ambitious program, hence it makes sense to consider special classes of manifolds and Lie groups. Under such hypotheses, one may be able to deduce rigidity results. For instance, Delzant proved in \cite{Delzant} that a compact symplectic manifold of dimension $2n$ endowed with the effective action of an $n$-torus is a combinatorial object, since it is given by a certain kind of polytope. In odd dimension, the structure of toric contact manifolds has been investigated by Lerman in \cite{Lerman}, showing a pattern akin to the symplectic one, up to a few exceptional cases.

Apart from contact manifolds, cosymplectic manifolds can be considered as the odd-dimensional analogue of symplectic manifolds. They were first defined by Libermann in \cite{Libermann} and have been intensively investigated since, see for instance \cite{BFM,CDNY,CdLM,CF,GMP,Li}.

In this short note we study Hamiltonian torus actions on cosymplectic manifolds. In Section \ref{sec:toric} we give the natural definition of toric actions on cosymplectic manifolds and prove some structural results: toric cosymplectic manifolds arise as mapping tori of toric symplectic manifolds with a $T$-equivariant symplectomorphism. By a result of Pinsonnault \cite{Pinsonnault}, this implies that, as smooth manifolds, they are products of toric symplectic manifolds with a circle. A different proof of this fact in the $K$-cosymplectic case is given in Section \ref{Toric:K:cosymplectic}.

In the Poisson world, cosymplectic manifolds appear naturally as the vanishing locus of $\Pi^{n+1}$, where $\Pi$ is the Poisson bivector on a $(2n+2)$-dimensional $b$-Poisson manifold. Equivalently, they appear as the locus where the $b$-symplectic form blows up on a $(2n+2)$-dimensional $b$-symplectic (a.k.a.\ log symplectic) manifold, see \cite{FMTM,GMP1}. Compact toric $b$-symplectic manifolds have been investigated in \cite{GLPT,GMP2}. In particular, in \cite{GMP2}, the authors show that the effective action of a torus $T^{n+1}$ on a $(2n+2)$-dimensional $b$-symplectic manifold induces an effective action of a torus $T^n$ on the cosymplectic hypersurface, thus producing a compact toric cosymplectic manifold. Under the hypothesis that the symplectic foliation of this cosymplectic manifold admits a compact leaf, they show that the cosymplectic manifold is the product of a toric symplectic manifold and a circle, thus obtaining a result analogous to ours. With respect to their approach, we make the following remarks:
\begin{itemize}
 \item we do not assume the existence of a compact leaf in the symplectic foliation, we rather prove that there is always one; this follows from Corollary \ref{cor:kompaktes:blatt} and the proof of Theorem \ref{thm:toriccosymplectic};
 \item in \cite[Theorem 50]{GMP1}, the authors prove that every cosymplectic manifold arises as critical hypersurface of a $b$-Poisson manifold \emph{with boundary}. However, the theory of toric $b$-Poisson manifolds developed in \cite{GMP2} is valid only for \emph{closed} manifolds, hence it is not immediately clear how to recover results for arbitrary toric cosymplectic manifolds.
\end{itemize}

A previous preprint version of this paper contained a proof of a convexity theorem for the momentum map of a Hamiltonian action on a cosymplectic manifold. The referee pointed out to us that such a theorem did exist previously, see \cite[Theorem 4.4]{ZhenqiHe}.

\subsection*{Acknowledgements}
We would like to thank Eva Miranda and \'Alvaro Pelayo for useful discussions. The authors are also indebted to the anonymous referee for her/his help in improving the exposition of the paper and for pointing out to us relevant literature. The first author is supported by a Juan de la Cierva -- Incorporaci\'on Grant of Ministerio de Ciencia, Innovaci\'on y Universidades (Spain).

\section{Cosymplectic and Hamiltonian Lie group actions}

\begin{definition}
A $(2n+1)$-dimensional manifold is called \emph{cosymplectic} if it is endowed with a $1$-form $\eta$ and a $2$-form $\omega$, both closed, such that $\eta\wedge \omega^n$ is a volume form.
\end{definition}

Since the 1-form $\eta$ is closed, its kernel defines a codimension 1 foliation, denoted $\cF_\eta$. On a cosymplectic manifold $(M,\eta,\omega)$ the \emph{Reeb vector field} $R$ is uniquely determined by the conditions $\imath_R \eta = 1$ and $\imath_R \omega = 0$. It defines a 1-dimensional foliation $\cF_R$ of the cosymplectic manifold.

Simple examples of cosymplectic manifolds are provided by \emph{symplectic mapping tori} (see \cite{Li}): let $(L,\sigma)$ be a symplectic manifold and let $\varphi\colon L\to L$ be a symplectomorphism. Consider the $\bZ$-action on $L\times\bR$ generated by $(p,t)\mapsto (\varphi(p),t+1)$. The quotient space $L_\varphi=(L\times\bR)/\bZ$ is a smooth manifold and has a natural cosymplectic structure $(\eta,\omega)$, obtained by projecting to $L_\varphi$ the cosymplectic structure $(dt,\sigma)$ on $L\times\bR$. Moreover, one has a fiber bundle $L\to L_\varphi\to S^1$.

Let $(M,\eta,\omega)$ be a cosymplectic manifold and let $G$ be a Lie group acting smoothly on $M$. We denote the action of $g\in G$ on $M$ by $x\mapsto g\cdot x$.

\begin{definition}
The action is \emph{cosymplectic} if
\[
g^*\eta=\eta \quad \mathrm{and} \quad g^*\omega=\omega \qquad \forall g\in G\,;
\]
in this case one says that $G$ acts on $(M,\eta,\omega)$ by \emph{cosymplectomorphisms}.
\end{definition}

If $G$ acts by cosymplectomorphisms then all fundamental vector fields of the action are cosymplectic, i.e., for all $A\in \fg$ we have
\[
L_{\oL{A}}\eta=0 \quad \mathrm{and} \quad L_{\oL{A}}\omega=0\,,
\]
where 
\[
\oL{A}(x)=\frac{d}{dt}\Big|_{t=0}\exp(tA)\cdot x
\]
denotes the corresponding fundamental vector field. 

\begin{proposition}{\cite[Corollary 6.4]{Bazzoni-Goertsches}} \label{prop:etalinearform}
Let $G$ be a Lie group acting on a cosymplectic manifold $(M,\eta,\omega)$ by cosymplectomorphisms. Then $\eta$ induces a linear form $\eta_0\in\fg^*$.
\end{proposition}

\begin{definition}
Let $(M,\eta,\omega)$ be a cosymplectic manifold. A vector field $X$ on $M$ is \emph{Hamiltonian} if $\eta(X)=0$ and there exists $f\in C^\infty(M)$, invariant under the Reeb flow, such that $\imath_X\omega=df$. The function $f$ is the \emph{Hamiltonian} of the vector field $X$.
\end{definition}

Given $f\in C^\infty(M)$ invariant under the Reeb flow, the corresponding \emph{Hamiltonian vector field}, denoted $X_f$, is uniquely determined by the condition $\eta(X_f)=0$ and $\imath_{X_f}\omega=df$.

The following definition was introduced by Albert \cite{Albert}.

\begin{definition}
Let $(M,\eta,\omega)$ be a cosymplectic manifold and let $G$ be a Lie group acting on $M$ by cosymplectomorphisms. The action is \emph{Hamiltonian} if there exists a smooth map $\mu\colon M\to\fg^*$ such that
\[
\oL{A}=X_{\mu^A} \quad \forall A\in\fg\,,
\]
where $\mu^A\colon M\to\bR$ is defined by the rule $\mu^A(x)=\mu(x)(A)$ and $X_{\mu^A}$ is the corresponding Hamiltonian vector field. Furthermore, we require $\mu$ to be equivariant with respect to the natural coadjoint action of $G$ on $\fg^*$:
\[
\mu(g\cdot x)=g\cdot\mu(x)=\mu(x)\circ \mathrm{Ad}_{g^{-1}}\,.
\]
$\mu$ is called a \emph{momentum map} of the action.
\end{definition}

It follows from the definition that $\mu^A$ is invariant under the Reeb flow and that each component of the momentum map satisfies $d\mu^A=\imath_{\oL{A}}\omega$. Moreover $\eta(\oL{A})=0$ for every $A\in\fg$; this implies that, if $M$ is compact, the $\bR$-action given by the flow of the Reeb field is never Hamiltonian.

\section{Toric cosymplectic manifolds}\label{sec:toric}

As in the symplectic situation, one has an upper bound for the dimension of a torus that can act effectively and in a Hamiltonian fashion:
\begin{lemma} Consider an effective Hamiltonian action of a torus $T$ on a cosymplectic manifold $M$ of dimension $2n+1$. Then $\dim T \leq n$.
\end{lemma}
\begin{proof}
Let $\mu:M\to \ft^*$ be a momentum map of the action. Then, as $T$ is Abelian, by \cite[Section 7]{Bazzoni-Goertsches} for all $A,B\in \ft$
\[
0 = -\mu^{[A,B]} = \{ \mu^A,\mu^B\} = - \omega(\bar{A},\bar{B}),
\]
where $\{\cdot,\cdot\}$ is the Poisson bracket. The tangent spaces of the $T$-orbits are therefore isotropic subspaces of $\cF_\eta$, which can be at most $n$-dimensional.
\end{proof}

Motivated by the symplectic setting we thus define toric objects in the cosymplectic category as those with a Hamiltonian action of maximal dimension:

\begin{definition} A \emph{toric cosymplectic manifold} is a cosymplectic manifold $(M,\eta,\omega)$ of dimension $2n+1$ with an effective Hamiltonian $T^n$-action.
\end{definition}

\begin{example} \label{ex:toricmappingtori}
If $(L,\sigma,T,\mu)$ is a toric symplectic manifold, and $\varphi:L\to L$ is any $T$-equivariant symplectomorphism, consider the mapping torus $L_\varphi = L\times [0,r]/(p,0)\sim (\varphi(p),r)$, for some number $r>0$. Its natural cosymplectic structure is toric with respect to the $T$-action on the factor $L$.
\end{example}

\begin{proposition}\label{prop:momentummorsebott} Consider a Hamiltonian $T$-action on a compact cosymplectic manifold $M$, with momentum map $\mu\colon M\to \ft^*$. Then for generic $A\in \ft$, the component $\mu^A$ of the momentum map $\mu$ is a $T$-invariant Morse-Bott function with critical set $M^T$.
\end{proposition}

\begin{proof}
In the special case of an action on a $K$-cosymplectic manifold, where $T$ is the subtorus of the closure of the flow of the Reeb vector field whose Lie algebra is the kernel of $\eta$, this was shown in \cite[Proposition 8.4]{Bazzoni-Goertsches}. The proof in the general case is the same; the only difference is that in our situation we cannot identify the fixed point set $M^T$ with the union of the closed Reeb orbits (denoted $C$ in \cite{Bazzoni-Goertsches})\footnote{Note that the centered chain of equations in the proof of \cite[Proposition 8.4]{Bazzoni-Goertsches} is incorrect. The correct argument is as follows: one first chooses an adapted Riemannian metric $g$ for which the $T$-action is isometric. For this, one first averages over $T$ to obtain a $T$-invariant metric $g_{\mathcal D}$ on ${\mathcal D} = \ker \eta$, and then defines a $T$-invariant Riemannian metric by $\tilde g = g_{{\mathcal{D}}} + \eta\otimes \eta$. Then one continues as in the third paragraph of the proof of \cite[Proposition 2.8]{Bazzoni-Goertsches}. The metric $g$ produced there is then automatically $T$-invariant. Using this metric $g$, the correct version of the computation in \cite[Proposition 8.4]{Bazzoni-Goertsches} reads ${\mathrm{Hess}}_{\mu^A}(v,w)(p)=V(\omega(\bar A,W))=V(g(\bar A,\phi W))=g(\nabla_V \bar A,\phi W)+g(\bar A,\nabla_V \phi W)=g(\nabla_V\bar A,\phi W)=\omega(\nabla_V\bar A,W)$. The missing factor $2$ is irrelevant for what follows.}. The proposition is also a special case of \cite[Theorem 3.4.6]{LinSjamaar}.
\end{proof}

For our main result we need a sufficient condition for the compactness of the leaves of $\cF_\eta$, which is essentially a distillation of the proof of \cite[Theorem 1]{Tischler}.

\begin{lemma}\label{kompaktes:blatt}
Let $M$ be a compact manifold. For any closed, nowhere vanishing $\eta'\in \Omega^1(M)$ such that the cohomology class $[\eta']$ is a real multiple of an integer class, the leaves of $\cF_{\eta'}$ are compact.
\end{lemma}
\begin{proof}
Via the usual identification of $H^1(M,\bZ)$ with homotopy classes of maps $M\to S^1$, we can write $c\eta' = f^*(dt)+dg$, for maps $f:M\to S^1$ and $g:M\to \bR$ and a nonzero real constant $c$ which is irrelevant as the kernel of $c\eta'$ equals the kernel of $\eta'$. Denoting by $\pi:\bR\to S^1$ the projection, we can write the right-hand side as $(f + \pi\circ g)^*(dt)$, where $+$ is the group structure in $S^1$. This in particular means that $h\coloneqq f + \pi\circ g$ is a submersion. Hence, it defines a fibration $h:M\to S^1$ which coincides with the foliation defined by $\eta'$.
\end{proof}

The lemma shows that if the foliation $\cF_\eta$ has nonclosed leaves, then no multiple of the class $[\eta]$ belongs to $H^1(M;\bZ)$. 

\begin{corollary}\label{cor:kompaktes:blatt}
Let $(M,\eta,\omega)$ be a compact cosymplectic manifold with $b_1(M)=1$. Then the leaves of $\cF_\eta$ are compact.
\end{corollary}
\begin{proof}
Since $b_1(M)=1$ and $M$ is compact, it follows that $H^1(M;\bR)\cong \langle [\eta]\rangle$. Since $H^1(M;\bZ)/\mathrm{torsion}\hookrightarrow H^1(M;\bR)$, a suitable multiple of $\eta$ defines an integer cohomology class, and we can apply Lemma \ref{kompaktes:blatt}.
\end{proof}

\begin{thm} \label{thm:toriccosymplectic}
Let $(M^{2n+1},\eta,\omega)$ be a compact toric cosymplectic manifold.
\begin{enumerate}
\item $b_{2k}(M) = b_{2k+1}(M)$ for all $k=0,\ldots,n$. In particular, $b_1(M) = b_{2n}(M) = 1$. Further, the leaves of the foliation $\cF_\eta$ are (compact) toric symplectic manifolds.
\item $M$ is cosymplectomorphic to the mapping torus of a $T$-equivariant symplectomorphism of a toric symplectic manifold, as in Example \ref{ex:toricmappingtori}.
\end{enumerate}
\end{thm}
\begin{proof}
By Proposition \ref{prop:momentummorsebott} a generic component $\mu^A\colon M\to \ft^*$ is a Morse-Bott function with critical set $M^T$. In this situation it is known that the $T$-action is equivariantly formal in the sense that the equivariant cohomology $H^*_T(M)$ is a free module over $S(\ft^*)$. Further, $\mu^A$ is automatically an (equivariantly) perfect Morse-Bott function, by the Atiyah-Bott Lemma \cite[Proposition 13.4]{Atiyah-Bott} (see also \cite[Proposition 4]{Duflot}). Thus, the Poincar\'e polynomial $P_t(M) = \sum_i t^i b_i(M)$ of $M$ is given by
\[
P_t(M) = \sum_B t^{\lambda_B} P_t(B),
\]
where $B$ runs over the connected components of $M^T$. In our situation, every component of $M^T$ is a circle $S^1$, hence $P_t(B) = 1+t$. Moreover, all indices $\lambda_B$ are even. This implies that $b_{2k}(M) = b_{2k+1}(M)$ for all $k=0,\ldots,n$, and in particular $b_1(M) = 1$.

By Corollary \ref{cor:kompaktes:blatt}, this implies that the leaves of $\cF_\eta$ are compact. At this point, one can apply the same technique used in the proof of Proposition 15 and in the discussion leading to Corollary 16 in \cite{GMP}; the family of diffeomorphisms alluded to there is given, in our situation, by the flow of the Reeb field. We conclude that $M$ is cosymplectomorphic to the mapping torus $L_\varphi$ of the symplectomorphism $\varphi\colon L\to L$ which is obtained as the restriction of a ``time-$r$ map'' of the Reeb flow to a leaf $L$. As $T$ commutes with the Reeb flow, $\varphi$ is $T$-equivariant.
\end{proof}

The second statement of Theorem \ref{thm:toriccosymplectic} identifies a compact toric cosymplectic manifold $M$ with the mapping torus of a $T$-equivariant symplectomorphism $\varphi$ of a toric symplectic manifold $(L,\sigma,T,\mu)$. Such a mapping torus is, as a smooth manifold, isomorphic to the product $L\times S^1$ if and only if $\varphi$ is isotopic to the identity. By a result of Pinsonnault \cite[Proposition 3.21 (1)]{Pinsonnault} this is always the case: there it is shown that the centralizer of $T$ in the symplectomorphism group is connected and contained in the group of Hamiltonian automorphisms. We obtain:

\begin{corollary}\label{cor:toriccosymplectic}
A compact toric cosymplectic manifold is, as a smooth manifold, the product of a toric symplectic manifold with a circle.
\end{corollary}


\begin{remark}
The reference of Pinsonnault was pointed out to us by the referee. A previous version of this paper obtained the same result by using an argument due to Lerman and Tolman, see \cite[Proposition 7.3]{Lerman-Tolman}.
\end{remark}

\begin{remark}
 Not every compact coK\"ahler manifold is toric, even when the first Betti number is 1. For instance, the compact 3-dimensional coK\"ahler manifold constructed in \cite{CdLM} has $b_1=1$ but is not homeomorphic to the product $S^2\times S^1$, hence it is not toric.
\end{remark}

\section{Toric K-cosymplectic manifolds}\label{Toric:K:cosymplectic}

In \cite{Bazzoni-Goertsches} we introduced the notion of $K$-cosymplectic manifold, in analogy to the notion of $K$-contact manifolds in contact geometry.

\begin{definition}
A \emph{K-cosymplectic manifold} is a cosymplectic manifold $(M,\eta,\omega)$ such that the Reeb field $R$ is Killing with respect to some Riemannian metric.
\end{definition}

If $(M,\eta,\omega)$ is a K-cosymplectic manifold equipped with a Riemannian metric with respect to which the Reeb field is Killing, then $\cF_R$ is a Riemannian foliation. Let $(L,\sigma)$ be a symplectic manifold with a compatible Riemannian metric and let $\varphi\colon L\to L$ be a symplectomorphism which is also an isometry. It is easy to see that $L_\varphi$ has a natural K-cosymplectic structure (see \cite[Proposition 2.12]{Bazzoni-Goertsches}).

Not every cosymplectic manifold $(M,\eta,\omega)$ admits a metric which makes it K-cosymplectic. We provide two classes of examples.

\begin{example}
Let $(L,\sigma)$ be a compact symplectic manifold and let $\varphi\colon L\to L$ be a symplectomorphism. Suppose that there is a metric on $(L_\varphi,\eta,\omega)$ such that the Reeb field $R$ is Killing. As proved in \cite[Proposition 2.8]{Bazzoni-Goertsches}, this is equivalent to the existence of an \emph{adapted} Riemannian metric $h$ for which $R$ is Killing, namely one for which $R$ is orthogonal to $\cF_\eta$. In our case, $\cF_\eta$ is just the fiber of the mapping torus bundle $L\to L_\varphi\to S^1$. The Reeb field $R$ is the projection to $L_\varphi$ of the vector field $\frac{\partial}{\partial t}$ on $\bR$, hence it projects further to a nowhere vanishing vector field on $S^1$. Since $R$ is Killing, its flow consists of isometries and its time-1 map generates the structure group of the mapping torus bundle; the latter is the cyclic group generated by $\varphi$, according to \cite[Propostion 6.4]{Bazzoni-Oprea}. Consider the sequence $\{\varphi^n\}_{n\in\bN}$; since $\mathrm{Isom}(L_\varphi,h)$ is a compact Lie group, by Myers-Steenrod, there exists a convergent subsequence $\{\varphi^{n_k}\}_{k\in\bN}$; call $\bar{\varphi}$ the limit isometry; since $\varphi$ leaves every leaf invariant, $\left\{\varphi^{n_k}\big|_L\right\}_k$ converges to $\bar{\varphi}\big|_L$. Consider now the standard symplectic form on the torus $T^2$ and the symplectomorphism $\varphi_A\colon T^2\to T^2$ covered by the linear map
\[
A=\begin{pmatrix}
   2 & 1\\1 & 1
  \end{pmatrix}\,.
\]
of $\bR^2$. For different $n\in\bN$, the actions of $\varphi_A^n$ on $H^1(T^2;\bZ)\cong\bZ^2$ are never equivalent, hence the sequence $\{\varphi_A^n\}_n$ does not admit any convergent subsequence. Hence, $T^2_{\varphi_A}$ does not admit any metric which makes the cosymplectic structure K-cosymplectic.
\end{example}

%

\begin{example}
A \emph{cosymplectic structure} on a Lie algebra $\fn$ with $\dim\fn=2n+1$ is a pair $(\eta,\omega)$ with $\eta\in\fn^*$, $\omega\in\Lambda^2\fn^*$ such that $d\eta=0=d\omega$ and $\eta\wedge\omega^n\neq 0$, $d$ being the Chevalley-Eilenberg differential. Assume that there is a scalar product $h$ on $\fn$ such that $R\in\fn$, the Reeb field of $(\eta,\omega)$, is Killing. Then $R$ is parallel with respect to the Levi-Civita connection of $h$, see \cite[Corollary 2.4]{Bazzoni-Goertsches}. In this case, one checks that $\fn$ splits as a direct sum $\fn'\oplus\langle R\rangle$, where $\fn'=\ker\eta$. According to \cite[Proposition 5.21]{Bazzoni-Marrero}, there exist nilpotent Lie algebras, not a direct sum, which admit cosymplectic structures; on such Lie algebras, one finds no scalar product for which the Reeb field is Killing. A \emph{nilmanifold} is the compact quotient of a connected, simply connected, nilpotent Lie group $N$ by a lattice $\Gamma$. A geometric structure on a nilmanifold $\Gamma\backslash N$ is \emph{invariant} if it comes from a left-invariant geometric structure on $N$, that is, from a geometric structure on $\fn=\textrm{Lie}(N)$. Thus this construction provides examples of (invariant) cosymplectic structures on nilmanifolds which are not K-cosymplectic.
\end{example}

By Corollary \ref{cor:toriccosymplectic}, a compact toric cosymplectic manifold is the product of a compact symplectic toric manifold and a circle. In this section we provide an alternative proof of this corollary in the case of toric $K$-cosymplectic manifolds. 

\begin{thm}\label{thm:2}
Let $(M,\eta,\omega)$ be a compact toric K-cosymplectic manifold of dimension $2n+1$. Then $M$ is equivariantly cosymplectomorphic to a mapping torus $L_t = L\times [0,r]/(p,0)\sim (tp,r)$, where $L$ is a toric symplectic manifold, acted on by an $n$-dimensional torus $T$, $r>0$ some number, and $t$ an element of $T$, considered as a symplectomorphism $t:L\to L$. In particular, $M$ is diffeomorphic to $L\times S^1$.
\end{thm}
\begin{proof}
We know that $M$ is $T$-equivariantly cosymplectomorphic to a mapping torus of some $T$-equivariant symplectomorphism $\varphi\colon L\to L$, where $L$ is a toric symplectic manifold. We have to show that $\varphi$ is given by an element of $T$. For this, we note that as $M$ is $K$-cosymplectic, the Reeb flow consists of isometries with respect to some auxiliary metric, hence its closure is a compact Lie subgroup. It follows that the  subgroup $G$ of the isometry group of $M$ defined as the closure of the group generated by $T$ and $\varphi$ is a compact Abelian Lie group acting on $L$ by symplectomorphisms. By \cite[Lemma 7.1]{Masuda}, the only elements in $N_G(T)$ acting trivially on $T$ are the elements of $T$ themselves. As $G$ is Abelian, it follows that $G=T$, i.e., $\varphi\in T$.

The last assertion follows because the diffeomorphisms of $L$ given by the elements of $T$ are isotopic to the identity.
\end{proof}

\begin{corollary}
A compact toric $K$-cosymplectic manifold $(M,\eta,\omega)$ carries an adapted $T$-invariant coK\"ahler metric.
\end{corollary}
\begin{proof}
By Theorem \ref{thm:2}, $M$ is equivariantly cosymplectomorphic to a symplectic mapping torus $L_t$, where $(L,\sigma,T,\mu)$ is a compact symplectic toric manifold and $t\in T$ is a symplectomorphism; hence $L$ admits a compatible K\"ahler structure, since it is obtained by symplectic reduction of $\bC^N$. Any element of $T$ acts by K\"ahler automorphisms of $L$, thus $M=L_t$ admits a natural coK\"ahler metric (see Proposition 2.12 and Remark 2.13 in \cite{Bazzoni-Goertsches}).
\end{proof}

\begin{example}
Let $L$ be a toric symplectic manifold, equipped with a compatible K\"ahler metric, and $t\in T$ such that $\{t^n\mid n\in \bZ\}$ is dense in $T$. Then $T$ acts by isometries on $L$, and the mapping torus $L_t$, as a $K$-cosymplectic manifold, has only finitely many closed Reeb orbits. They are exactly those Reeb orbits that pass through the finite fixed point set $L^T$.
\end{example}

\begin{example}
Consider $S^2$, with the standard circle action. Then an equatorial Dehn twist $\varphi\colon S^2\to S^2$ is an $S^1$-equivariant symplectomorphism. The natural cosymplectic structure on the associated mapping torus is toric, but not $K$-cosymplectic. In fact, if it was, then the time-1 map of the Reeb flow would be an isometry fixing a nonempty open subset but not the whole manifold, which is impossible. Note that $\varphi$ is isotopic to the identity, so that this mapping torus is, as a smooth manifold, diffeomorphic to $S^2\times S^1$.
\end{example}

In \cite[Theorem 4.3]{Bazzoni-Goertsches} we showed that for a $K$-cosymplectic manifold, the basic cohomology $H^*(M;\cF_R)$ of the Reeb foliation encodes the same information as ordinary de Rham cohomology $H^*(M;\bR)$: we have an equality 
\[ 
H^*(M;\bR) = H^*(M;\cF_R)\otimes \Lambda\langle [\eta]\rangle,
\]
i.e.,
\begin{equation}\label{eq:basicderhamcohom}
H^p(M;\bR) = H^p(M;\cF_R) \oplus [\eta]\wedge H^{p-1}(M;\cF_R)
\end{equation}
for all $p=0,\ldots,2n+1$. In particular,
\[
H^1(M;\bR) = H^1(M;\cF_R) \oplus \bR\cdot [\eta].
\]
\begin{corollary}
For a toric $K$-cosymplectic manifold $M$ we have $H^{odd}(M;\cF_R) = 0$.
\end{corollary}
\begin{proof}
The equalities \eqref{eq:basicderhamcohom} imply the relation $b_p(M) = b_p(M,\cF_R) + b_{p-1}(M,\cF_R)$ between the Betti and basic Betti numbers of $M$. Combining this with the equalities $b_{2k}(M) = b_{2k+1}(M)$ shown in Theorem \ref{thm:toriccosymplectic} the claim follows by induction.
\end{proof}

\section{Deformations}

In this section we recall the deformations of type I and II of cosymplectic structures. We believe that such deformations are worth investigating: both of them preserve the property of a cosymplectic manifold of being toric; those of type II allow to construct ``genuine'' cosymplectic manifolds, that is, cosymplectic manifolds which are not cosymplectomorphic to symplectic mapping tori.


\subsection{Type I deformations}

In \cite{Bazzoni-Goertsches} we introduced deformations of type I of almost contact (metric) structures. Let $(M,\eta,\omega)$ be a cosymplectic manifold and let $\theta\in\mathfrak{X}(M)$ be a cosymplectic vector field such that $1+\eta(\theta)>0$. Define
\[
\eta'=\frac{\eta}{1+\eta(\theta)} \quad \textrm{and}
\quad \omega'=\frac{\omega+\imath_\theta\omega\wedge\eta'}{1+\eta(\theta)}\,.
\]
Then $(M,\eta',\omega')$ is a cosymplectic manifold. Under appropriate compatibility conditions, type I deformations preserve K-cosymplectic structures, see \cite[Proposition 6.2]{Bazzoni-Goertsches}.

Suppose $(M,\eta,\omega)$ is a cosymplectic manifold endowed with the Hamiltonian action of a torus $T$. Given $\theta\in\ft$, we can use the corresponding fundamental vector field $\bar{\theta}$ to perform a type I deformation of the cosymplectic structure, giving $(M,\eta',\omega')$. Since the fundamental fields of the action belong to $\cF_\eta$, and this foliation is preserved under a type I deformation, it makes sense to ask whether the given torus action is also Hamiltonian for the deformed structure. This is indeed the case:

\begin{proposition}
The $T$-action on the deformed cosymplectic structure is Hamiltonian with momentum map $\mu'\colon M\to\ft^*$,
\[
\mu'=\frac{1}{1+\eta(\bar{\theta})}\mu\,,
\]
where $\theta\in\ft$ induces the vector field $\bar{\theta}$ used to deform.
\end{proposition}
\begin{proof}
Let $X$ be an element in $\ft$ and denote by $\bX$ the corresponding fundamental vector field on $M$. We compute
  \begin{align*}
   \imath_{\oL{X}}\omega'&=\frac{1}{1+\eta(\bar\theta)}\left(\imath_{\bX}(\omega+\imath_{\bar\theta}\omega\wedge\eta')\right)=
   \frac{1}{1+\eta(\bar\theta)}\left(d\mu^X+\omega(\bar\theta,\bX)\eta'\right)\\
   &=\frac{1}{1+\eta(\bar\theta)}\left(d\mu^X-d\mu^X(\bar\theta)\eta'\right)=\frac{1}{1+\eta(\bar\theta)}\left(d\mu^X-\bar\theta(\mu^X)\eta'\right)\,.
  \end{align*}
  Let us fix $x\in M$; the integral curve of $\bar{\theta}$ passing through $x$ comes from the 1-parameter subgroup of $\theta\in\ft$; by equivariance, for every $g\in T$ belonging to this 1-parameter subgroup we have
  \[
   \mu^X(g\cdot x)=\mu^X(x)\circ\Ad_{g^{-1}}=\mu^X(x)
  \]
  since, $T$ being abelian, the adjoint action is trivial. Thus $\mu^X$ is constant along the flow of $\bar{\theta}$ and the second summand vanishes, leaving us with
  \[
   \imath_{\oL{X}}\omega'=\frac{d\mu^X}{1+\eta(\bar\theta)}\,.
   \]
\end{proof}
Note that by Proposition \ref{prop:etalinearform}, the momentum map has changed only by a constant.

\subsection{Type II deformations}\label{sec:exdefII}
In \cite{Bazzoni-Goertsches} we introduced as well type II deformations of cosymplectic structures, in analogy with the contact setting. Consider a cosymplectic manifold $(M^{2n+1},\eta,\omega)$ with Reeb field $R$ and let $\beta\in\Omega^1(M)$ be an arbitrary closed basic form, that is, $d\beta=0$ and $\imath_R\beta=0$. Then $\eta'=\eta+\beta$ is again a closed 1-form on $M$ and since $\imath_R\beta=0=\imath_R\omega$, $\beta\wedge\omega^n=0$, hence $\eta'\wedge\omega=\eta\wedge\omega$ is a volume form on $M$ and $(M,\eta',\omega)$ is a new cosymplectic structure on $M$. The 2-form $\omega$ and the Reeb field of the deformed structure are the same, but the foliation $\cF_\eta$ has changed.

Suppose a torus $T$ acts in a Hamiltonian fashion on a compact cosymplectic manifold $(M,\eta,\omega)$ and let $\mu\colon M\to \ft^*$ be the corresponding momentum map. This implies, in particular, that $\bar{A}$ is tangent to $\cF_\eta$ for every $A\in\ft$. Suppose we perform a deformation of type II $(M,\eta' = \eta + \beta,\omega)$ of $(M,\eta,\omega)$, where $\beta$ is not only $R$-basic but also $T$-basic, i.e., $\imath_{\bar{A}} \eta = 0$ for all $A\in \ft$. The same $T$-action on $(M,\eta',\omega)$ is again Hamiltonian; as $\omega$ has not changed, the momentum map has not changed either.

The following general proposition shows that in order to keep the $T$-action cosymplectic it in fact suffices to assume that the closed one-form $\beta$ is $R$-basic and $T$-invariant, as $T$-horizontality is automatic:

\begin{proposition}\label{prop:TinvTbasic} Consider an action of a torus $T$ on a compact manifold $M$, with at least one fixed point. Let $\eta\in \Omega^1(M)$ be a $T$-invariant closed one-form on $M$. Then $\eta$ is $T$-basic.
\end{proposition}
\begin{proof} It suffices to show the claim for those $A\in \ft$ whose associated one-parameter subgroup of $T$ is closed, i.e., a circle $S^1$. The $T$-invariance of the closed one-form $\eta$ implies that $\imath_{\bar{A}}\eta$ is a constant, say $c$. Because the $T$-action has a fixed point, by computing this constant at a fixed point, we see that it has to vanish.
\end{proof}


We wish to show that in the $K$-cosymplectic setting we can always achieve, by a type II deformation, that the foliation $\cF_\eta$ has all leaves closed.

\begin{lemma}\label{lem:invbasicTR} Let $(M,\eta,\omega)$ be a compact cosymplectic manifold. Then the inclusion $\Omega(M,\cF_R)^T\to \Omega(M,\cF_R)$ of $T$-invariant, $R$-basic differential forms into $R$-basic differential forms is a quasi-isomorphism of differential graded algebras.
\end{lemma}
\begin{proof}
In \cite[\S 9.1, Theorem 1]{Onishchik} it is shown that the averaging operator $r\colon\Omega(M) \to \Omega(M)^T \subset \Omega(M)$ is homotopic to the identity, via a homotopy operator $k:\Omega(M)\to \Omega(M)$, i.e., $r-\id = d\circ k + k \circ d$. The averaging operator restricts to a well-defined map $\Omega(M,\cF_R)\to \Omega(M,\cF_R)^T$, and going through the proof, one verifies that also $k$ respects the condition of being $R$-basic. Thus, also $r\colon\Omega(M,\cF_R)\to \Omega(M,\cF_R)^T\subset \Omega(M,\cF_R)$ is homotopic to the identity, and one concludes as in the last paragraph of the proof in \cite{Onishchik}.
\end{proof}

We apply the lemma to the case in which $(M,\eta,\omega)$ is a compact K-cosymplectic manifold; then, by \eqref{eq:basicderhamcohom} we have $H^1(M;\bR) = H^1(M;\cF_R)\oplus \bR\cdot [\eta]$. Further, by Lemma \ref{lem:invbasicTR}, we can find an $R$-basic, $T$-invariant closed 1-form $\beta$ (which by Proposition  \ref{prop:TinvTbasic} is even $T$-basic) such that the class of $\eta' = \eta + \beta$ is a real multiple of an integer class. By Lemma \ref{kompaktes:blatt} the leaves of $\cF_{\eta'}$ are compact.

Finally we observe that we can use type II deformations also to produce examples of Hamiltonian actions on cosymplectic manifolds $M$ for which $\cF_\eta$ has nonclosed leaves, provided the first Betti number of $M$ is larger than one. Here is a very easy concrete example:

\begin{example}\label{ex:105}
 For $n\geq 1$ consider the manifold $M_n=\bC P^n\times T^2\times S^1$. Interpreted as the mapping torus of the identity on $\bC P^{n}\times T^2$, $M_n$ has an obvious cosymplectic structure $(\eta,\omega)$, where $\eta$ is the length 1-form on $S^1$ and $\omega$ is the product symplectic structure on $\bC P^n\times T^2$. The torus $T=T^n$ acts on $\bC P^n$ in the usual way,
 \[
  (t_1,\ldots,t_n)\cdot [z_0:\ldots:z_n]=[z_0:t_1z_1:\ldots:t_nz_n]\,. 
 \]
 and we consider the induced $T$-action on $M_n$ acting only on the first factor.
 
This action is clearly Hamiltonian in the cosymplectic sense. If $x_1,x_2\in\Omega^1(T^2)$ generate the cohomology of $T^2$, we have $H^1(M_n;\bR)=\langle [x_1],[x_2],[\eta]\rangle$. We now perform a deformation of type II to the cosymplectic structure by setting $\eta'=\eta+\ve_1x_1+\ve_2x_2$; in particular, if we choose $\ve_1$ and $\ve_2$ to be rationally independent, the leaves of the foliation $\cF_{\eta'}$ are the product of $\bC P^n$ with an irrational codimension 1 foliation on $T^2\times S^ 1$, hence are noncompact in $M_n$. However, $T$ acts in a Hamiltonian fashion on the deformed structure $(M_n,\eta',\omega)$, by the above discussion, and the image of the momentum map for the action on the deformed structure coincides with the original one, which is well-known to be the standard simplex in $\bR^n$.
 \end{example}

\bibliographystyle{plain}
\bibliography{bibliography}

\begin{thebibliography}{10}

\bibitem{Albert}
Claude Albert.
\newblock Le th\'eor\`eme de r\'eduction de {M}arsden-{W}einstein en
  g\'eom\'etrie cosymplectique et de contact.
\newblock {\em J. Geom. Phys.}, 6(4):627--649, 1989.

\bibitem{Atiyah-Bott}
Michael~F. Atiyah and Raoul Bott.
\newblock The {Y}ang-{M}ills equations over {R}iemann surfaces.
\newblock {\em Philos. Trans. Roy. Soc. London Ser. A}, 308(1505):523--615,
  1983.

\bibitem{BFM}
Giovanni Bazzoni, Marisa Fern\'andez, and Vicente Mu\~noz.
\newblock Non-formal co-symplectic manifolds.
\newblock {\em Trans. Amer. Math. Soc.}, 367(6):4459--4481, 2015.

\bibitem{Bazzoni-Goertsches}
Giovanni Bazzoni and Oliver Goertsches.
\newblock {$K$}-cosymplectic manifolds.
\newblock {\em Ann. Global Anal. Geom.}, 47(3):239--270, 2015.

\bibitem{Bazzoni-Marrero}
Giovanni Bazzoni and Juan~Carlos Marrero.
\newblock On locally conformal symplectic manifolds of the first kind.
\newblock {\em Bull. Sci. Math.}, 143:1--57, 2018.

\bibitem{Bazzoni-Oprea}
Giovanni Bazzoni and John Oprea.
\newblock On the structure of co-{K}\"{a}hler manifolds.
\newblock {\em Geom. Dedicata}, 170:71--85, 2014.

\bibitem{CDNY}
Beniamino Cappelletti-Montano, Antonio De~Nicola, and Ivan Yudin.
\newblock A survey on cosymplectic geometry.
\newblock {\em Rev. Math. Phys.}, 25(10):1343002, 55, 2013.

\bibitem{CdLM}
Domingo Chinea, Manuel de~Le\'on, and Juan~C. Marrero.
\newblock Topology of cosymplectic manifolds.
\newblock {\em J. Math. Pures Appl. (9)}, 72(6):567--591, 1993.

\bibitem{CF}
Diego Conti and Marisa Fern\'andez.
\newblock Einstein almost cok\"ahler manifolds.
\newblock {\em Math. Nachr.}, 289(11-12):1396--1407, 2016.

\bibitem{Delzant}
Thomas Delzant.
\newblock Hamiltoniens p\'eriodiques et images convexes de l'application
  moment.
\newblock {\em Bull. Soc. Math. France}, 116(3):315--339, 1988.

\bibitem{Duflot}
Jeanne Duflot.
\newblock Smooth toral actions.
\newblock {\em Topology}, 22(3):253--265, 1983.

\bibitem{FMTM}
Pedro Frejlich, David Mart\'{i}nez~Torres, and Eva Miranda.
\newblock A note on the symplectic topology of {$b$}-manifolds.
\newblock {\em J. Symplectic Geom.}, 15(3):719--739, 2017.

\bibitem{GLPT}
Marco Gualtieri, Songhao Li, \'Alvaro Pelayo, and Tudor~S. Ratiu.
\newblock The tropical momentum map: a classification of toric log symplectic
  manifolds.
\newblock {\em Math. Ann.}, 367(3-4):1217--1258, 2017.

\bibitem{GMP}
Victor Guillemin, Eva Miranda, and Ana~Rita Pires.
\newblock Codimension one symplectic foliations and regular {P}oisson
  structures.
\newblock {\em Bull. Braz. Math. Soc. (N.S.)}, 42(4):607--623, 2011.

\bibitem{GMP1}
Victor Guillemin, Eva Miranda, and Ana~Rita Pires.
\newblock Symplectic and {P}oisson geometry on {$b$}-manifolds.
\newblock {\em Adv. Math.}, 264:864--896, 2014.

\bibitem{GMP2}
Victor Guillemin, Eva Miranda, Ana~Rita Pires, and Geoffrey Scott.
\newblock Toric actions on {$b$}-symplectic manifolds.
\newblock {\em Int. Math. Res. Not. IMRN}, (14):5818--5848, 2015.

\bibitem{ZhenqiHe}
Zhenqi He.
\newblock {\em Odd Dimensional Symplectic Manifolds}.
\newblock PhD thesis, Massachusetts Institute of Technology, 2010.
\newblock https://core.ac.uk/download/pdf/4424560.pdf.

\bibitem{Lerman}
Eugene Lerman.
\newblock Contact toric manifolds.
\newblock {\em J. Symplectic Geom.}, 1(4):785--828, 2003.

\bibitem{Lerman-Tolman}
Eugene Lerman and Susan Tolman.
\newblock Hamiltonian torus actions on symplectic orbifolds and toric
  varieties.
\newblock {\em Trans. Amer. Math. Soc.}, 349(10):4201--4230, 1997.

\bibitem{Li}
Hongjun Li.
\newblock Topology of co-symplectic/co-{K}\"ahler manifolds.
\newblock {\em Asian J. Math.}, 12(4):527--543, 2008.

\bibitem{Libermann}
Paulette Libermann.
\newblock Sur les automorphismes infinit\'esimaux des structures symplectiques
  et des structures de contact.
\newblock In {\em Colloque {G}\'eom. {D}iff. {G}lobale ({B}ruxelles, 1958)},
  pages 37--59. Centre Belge Rech. Math., Louvain, 1959.

\bibitem{LinSjamaar}
Yi~Lin and Reyer Sjamaar.
\newblock Convexity properties of presymplectic moment maps.
\newblock 2017.
\newblock https://arxiv.org/abs/1706.00520.

\bibitem{Masuda}
Mikiya Masuda.
\newblock Symmetry of a symplectic toric manifold.
\newblock {\em J. Symplectic Geom.}, 8(4):359--380, 2010.

\bibitem{Onishchik}
Arkadi~L. Onishchik.
\newblock {\em Topology of transitive transformation groups}.
\newblock Johann Ambrosius Barth Verlag GmbH, Leipzig, 1994.

\bibitem{Pinsonnault}
Martin Pinsonnault.
\newblock Maximal compact tori in the {H}amiltonian group of 4-dimensional
  symplectic manifolds.
\newblock {\em J. Mod. Dyn.}, 2(3):431--455, 2008.

\bibitem{Tischler}
David Tischler.
\newblock On fibering certain foliated manifolds over {$S^{1}$}.
\newblock {\em Topology}, 9:153--154, 1970.

\end{thebibliography}

\end{document}